\documentclass{article}
\usepackage{amssymb,amsmath,amsthm,graphicx}
\usepackage[all,color]{xy}

\def\qed{\hfill {\hbox{${\vcenter{\vbox{               %HOLLOW SQUARE
   \hrule height 0.4pt\hbox{\vrule width 0.4pt height 6pt
   \kern5pt\vrule width 0.4pt}\hrule height 0.4pt}}}$}}}

\newtheorem{theorem}{Theorem}
\newtheorem{lemma}[theorem]{Lemma}
\newtheorem{proposition}[theorem]{Proposition}
\newtheorem{corollary}[theorem]{Corollary}

\theoremstyle{definition}
\newtheorem{example}{Example}
\newtheorem{definition}{Definition}
\newtheorem{remark}{Remark}

\date{}

\title{\Large \textbf{Biquasile Colorings of Oriented Surface-Links}}

\author{Jieon Kim\footnote{Email:jieonkim@sci.osaka-cu.ac.jp. Supported by JSPS overseas post doctoral fellow Grant Number 15F15319.} \and
Sam Nelson\footnote{Email: Sam.Nelson@cmc.edu. Partially supported by Simons Foundation collaboration grant 316709}}

\begin{document}
\maketitle

\begin{abstract}
We introduce colorings of oriented surface-links by biquasiles using marked
graph diagrams. We use these colorings to define counting invariants and 
Boltzmann enhancements of the biquasile counting invariants for oriented 
surface-links. We provide examples to show that the invariants can distinguish
both closed surface-links and cobordisms and are sensitive to orientation.
\end{abstract}

\parbox{5.5in} {\textsc{Keywords:} Biquasiles, counting invariants, 
surface-links, marked graph diagrams

\smallskip

}

\section{\large\textbf{Introduction}}\label{I}

In \cite{dsn2}, a type of algebraic structure known as \textit{biquasiles} was 
introduced and used to define invariants of oriented classical knots and links
via counting colorings of graphs known as \textit{dual graph diagrams}. 
Dual graph versions of a generating set of oriented Reidemeister moves
were identified and used to motivate the biquasile axioms and to 
prove invariance of the set of biquasile colorings under these moves. 
Biquasiles can be understood as a special case of the ternary algebraic 
structures defined in \cite{NM}.

In \cite{Lo}, planar graphs with extra information known as \textit{marked 
graph diagrams} (also sometimes called \textit{marked vertex diagrams} or 
\textit{ch-diagrams}) were introduced as way of encoding knotted and linked
surfaces in $\mathbb{R}^4$, known as \textit{surface-links}.  In \cite{Yo} 
a set of moves on marked graph diagrams encoding ambient isotopy of surfaces in 
$\mathbb{R}^4$ analogous to the Reidemeister moves was proposed. In \cite{KJL}, 
generating sets of these \textit{Yoshikawa moves} which are necessary and 
sufficient for ambient isotopy of oriented surface-links were identified. 

In this paper we define biquasile counting invariants for oriented 
surface-links and use them to define new nonnegative integer-valued invariants
of oriented surface-links. We then enhance these invariants with 
\textit{Boltzmann weights} taking values in a commutative ring $R$ such that 
the multiset of Boltzmann weight values over the complete set of biquasile 
colorings of a marked graph diagram defines a stronger invariant of 
surface-links from which we can recover the biquasile counting invariant by
taking the cardinality of the multiset. In particular, these invariants
potentially provide obstructions to cobordism between classical knots and links.
As this paper was nearing completion, the authors learned that similar results
have been independently obtained and recently presented by Maciej Niebrzydowski
\cite{MN}.

The paper is organized as follows. In 
Section \ref{M} we review the basics of marked graph diagrams and 
surface-links. In Section \ref{SB} we review biquasiles 
and define the biquasile counting invariant for surface-links. We provide 
examples to show that biquasile colorings can distinguish surface-links and 
can detect orientation reversals.
In Section \ref{BW} we recall the biquasile Boltzmann weight enhancement
and extend it to the case of oriented surface-links with some examples.
As an application we show that these invariants can distinguish
non-isotopic cobordisms between links. We end in Section \ref{Q} with some 
open questions for future research.

\section{\large\textbf{Marked Graph Diagrams}}\label{M}

In this section, we review (oriented) marked graph diagrams representing 
surface-links.

\begin{definition}
A {\it marked graph} is a 4-valent graph in $\mathbb{R}^3$ each of whose 
vertices has a marker that looks like
\[
\xy (-5,5);(5,-5) **@{-},
(5,5);(-5,-5) **@{-},
(3,-0.2);(-3,-0.2) **@{-},
(3,-0.1);(-3,-0.1) **@{-},
(3,0);(-3,0) **@{-},
(3,0.1);(-3,0.1) **@{-},
(3,0.2);(-3,0.2) **@{-},
\endxy.\] Two marked graphs are said to be {\it equivalent} if they are ambient 
isotopic in $\mathbb{R}^3$
with keeping the rectangular neighborhoods of markers. A marked 
graph in $\mathbb{R}^3$ can be described by a link diagram on $\mathbb{R}^2$ 
with some $4$-valent vertices equipped with markers, called a {\it marked 
graph diagram}.
\end{definition}

%A {\it marked graph} is a spatial graph $G$ in $\mathbb{ R^3$ which satisfies the following:
%\begin{itemize}
% \item [(1)] $G$ is a finite regular graph with $4$-valent vertices, say $v_1, v_2, . . . , v_n$.
% \item [(2)] Each $v_i$ is a rigid vertex; that is, we fix a rectangular neighborhood $N_i$ homeomorphic to $\{(x, y)|-1 \leq x, y \leq 1\},$
%where $v_i$ corresponds to the origin and the edges incident to $v_i$ are represented by $x^2 = y^2$.
% \item [(3)] Each $v_i$ has a {\it marker}, which is the interval on $N_i$ given by $\{(x, 0)|-1 \leq x \leq 1\}$.
%\end{itemize}

\begin{definition}
An {\it orientation} of a marked graph $G$ in $\mathbb{R}^3$ is a choice of an 
orientation for each edge of $G.$
An orientation of a marked graph $G$ is said to be {\it consistent} if every 
vertex in $G$ looks like
\[
\includegraphics{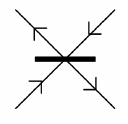} 
\quad\raisebox{0.25in}{or}\quad 
\includegraphics{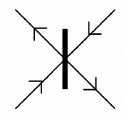}.\]
%\xy (-5,5);(5,-5) **@{-},
%(5,5);(-5,-5) **@{-},
%(3,3.2)*{\llcorner},
%(-3,-3.4)*{\urcorner},
%(-2.5,2)*{\ulcorner},
%(2.5,-2.4)*{\lrcorner},
%(3,-0.2);(-3,-0.2) **@{-},
%(3,-0.1);(-3,-0.1) **@{-},
%(3,0);(-3,0) **@{-},
%(3,0.1);(-3,0.1) **@{-},
%(3,0.2);(-3,0.2) **@{-}.
%\endxy\] 
A marked graph $G$ in $\mathbb{R}^3$ is said to be {\it orientable} if 
$G$ admits a consistent orientation. Otherwise, it is said to be 
{\it non-orientable}.
\end{definition}

\begin{definition}
By an {\it oriented marked graph}, we mean an orientable marked graph in
$\mathbb{R}^3$ with a fixed consistent orientation. Two oriented marked graphs 
are said to be {\it equivalent} if they are ambient isotopic in $\mathbb{R}^3$
with keeping the rectangular neighborhood of the marker and consistent orientation.
\end{definition}

For $t \in \mathbb{R},$ we denote by $\mathbb{R}^3_t$ the hyperplane of 
$\mathbb{R}^4$ whose fourth coordinate is equal to $t \in \mathbb{R}$, i.e., 
$\mathbb{R}^3_t = \{(x_1, x_2, x_3, x_4) \in
\mathbb{R}^4~|~ x_4 = t \}$. A surface-link 
$\mathcal{L}\subset \mathbb{R}^4=\mathbb{R}^3 \times \mathbb{R}$ can be 
described in terms of its {\it cross-sections} $\mathcal{L}_t=\mathcal{L} 
\cap \mathbb{R}^3_t, ~ t \in \mathbb{R}$ (cf. \cite{Fox}).

It is known (\cite{KSS,Lo}) that any surface-link $\mathcal{L}$ is equivalent 
to a surface-link $\mathcal{L}'$ such that the projection 
$p_{\mathcal{L}'}:\mathcal{L}' \to \mathbb{R}$ satisfies the following conditions:
\begin{itemize}
\item[(1)]A surface-link $\mathcal{L}'$ has finitely many critical points and 
all critical points are non-degenerate.
\item[(2)] All the index 0 critical points (minimal points) are in 
$\mathbb{R}^3_{-1}$.
\item[(3)] All the index 1 critical points (saddle points) are in 
$\mathbb{R}^3_{0}$.
\item[(4)] All the index 2 critical points (maximal points) are in 
$\mathbb{R}^3_{1}$.
\end{itemize}
We call $\mathcal{L}'$ a {\it normal form } of $\mathcal{L}.$

Let $\mathcal{L}$ be a surface-link in $\mathbb{R}^4$, and $\mathcal{L}'$ 
a normal form of $\mathcal{L}.$ Then $\mathcal{L}'_0$ is a spatial $4$-valent 
regular graph in $\mathbb{R}^3_0$. We give a marker at each $4$-valent vertex 
(saddle point) that indicates how the saddle point opens up above as 
illustrated.
\[\includegraphics{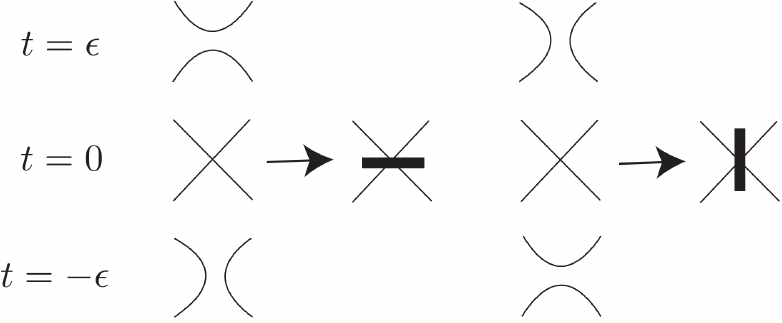}\]
%in 
%Fig.~\ref{sleesan2:fig1}. 
We choose an orientation for each edge of $\mathcal{L}'_0$
that coincides with the induced orientation on the boundary of 
$\mathcal{L}' \cap \mathbb{R}^3 \times (-\infty, 0]$ from the orientation of
$\mathcal{L}'$.
The resulting oriented marked graph $G$ is called an {\it oriented marked 
graph} of $\mathcal{L}$.
As usual, $G$ is described by a link diagram $D$ with rigid marked vertices. 
Such a diagram $D$ is called an {\it oriented marked
graph diagram} or an {\it oriented ch-diagram} (cf. \cite{So}) of $\mathcal{L}$.

Let $D$ be an oriented marked graph diagram. We obtain two links $L_-(D)$ and 
$L_+(D)$ from $D$ by replacing each marked vertex in $D$ as shown. 
\[\includegraphics{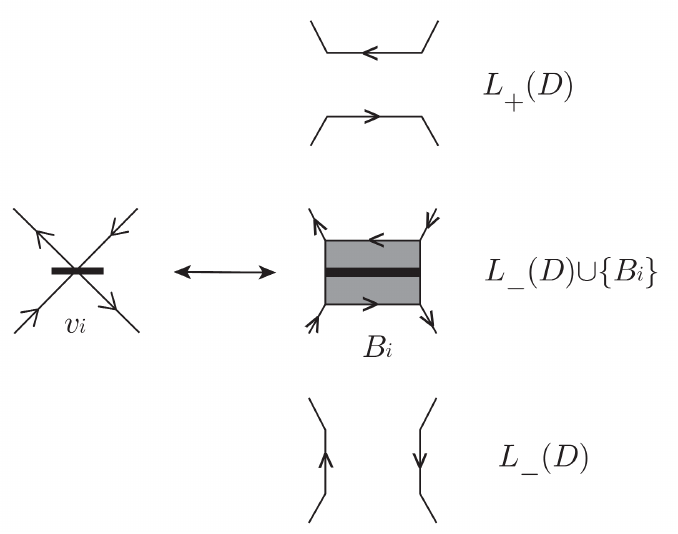}\]
%in Fig.~\ref{fig-orbd}. 
The links $L_-(D)$ and $L_+(D)$ are also called the {\it negative resolution} 
and the {\it positive resolution} of $D$, respectively. Conversely, we obtain
an oriented surface-link by replacing a 
neighborhood of each marked vertex $v_i~ (1 \leq i \leq n)$ with an oriented 
band $B_i$ as illustrated.
\[\scalebox{0.8}{\includegraphics{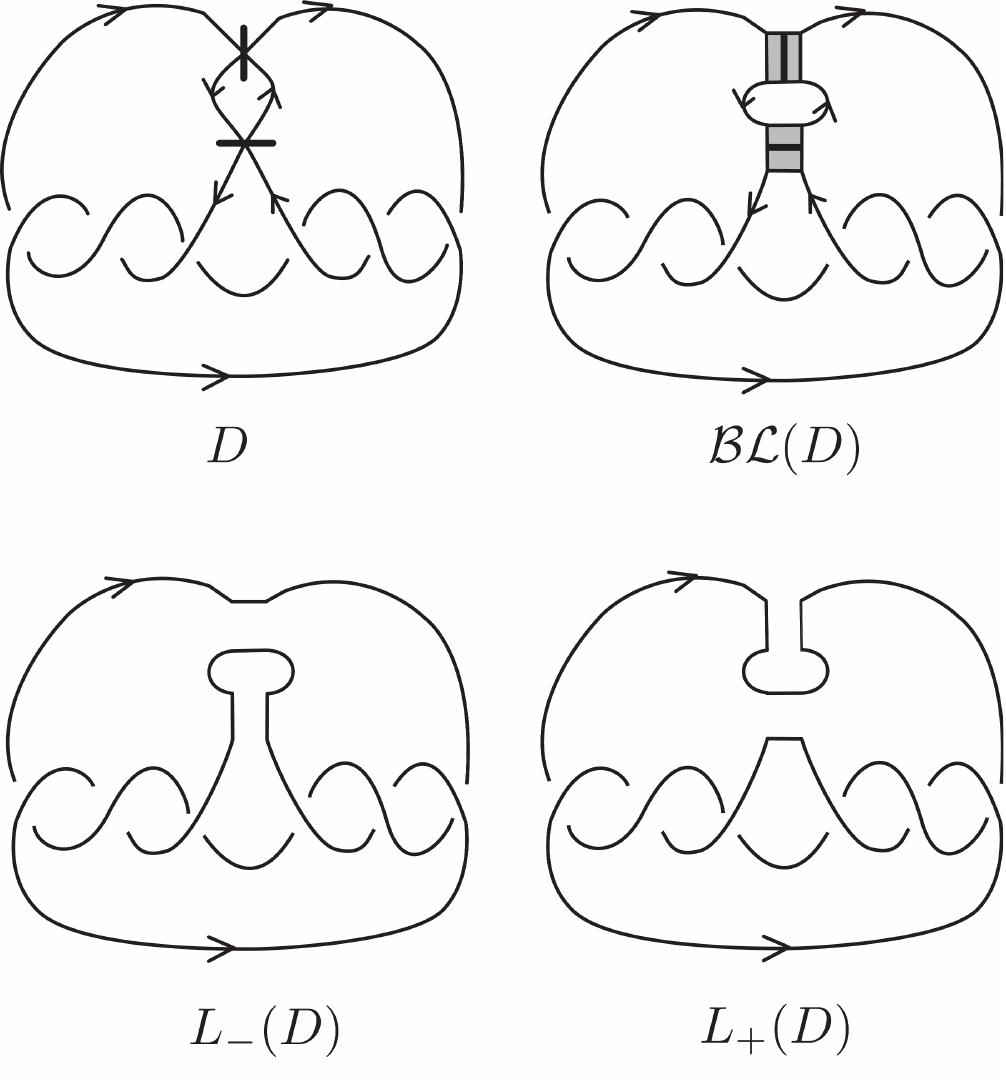}}\]
% in Fig.~\ref{fig-orbd}. 
Denote the disjoint union $B_1\sqcup\cdots\sqcup B_n$ of bands by 
$\mathcal B(D).$
A marked graph diagram $D$ is said to be {\it admissible}
if both resolutions $L_-(D)$ and $L_+(D)$ are trivial link diagrams.

%\begin{center}
%
%\resizebox{0.50\textwidth}{!}{%
%
%\includegraphics{spun-mgraph.eps} }
%
%\caption{An admissible marked graph diagram}
%
%\label{spun-mgraph-res}
%
%\end{center}

Let us describe how to construct an oriented surface-link 
$\mathcal{L}$ in $\mathbb{R}^4$ from any given admissible marked graph 
diagram up to equivalence \cite{Ka,KSS,Kaw,Yo}.
Let $\Delta_1,\ldots,\Delta_a \subset \mathbb{R}^3$ be mutually disjoint 
$2$-disks with $\partial(\cup_{j=1}^a\Delta_j)= L_+(D)$, and let 
$\Delta_1',\ldots,\Delta_b' \subset \mathbb{R}^3$ be mutually disjoint 
$2$-disks with $\partial(\cup_{k=1}^b\Delta_k')= L_-(D)$. We define a 
surface-link $\mathcal S(D) \subset \mathbb{R}^3 \times \mathbb{R} = 
\mathbb{R}^4 $ by
\begin{equation*}%\label{defn-surf-1}
(\mathbb{R}^3_t, \mathcal{S}(D) \cap \mathbb{R}^3_t)=\left\{%
\begin{array}{ll}
(\mathbb R^3, \phi) & \hbox{for $t > 1$,}\\
(\mathbb R^3, L_+(D) \cup (\cup_{j=1}^a\Delta_j)) & \hbox{for $t = 1$,} \\
(\mathbb R^3, L_+(D)) & \hbox{for $0 < t < 1$,} \\
(\mathbb R^3, L_-(D) \cup (\cup_{i=1}^n B_i)) & \hbox{for $t = 0$,} \\
(\mathbb R^3, L_-(D)) & \hbox{for $-1 < t < 0$,} \\
(\mathbb R^3, L_-(D) \cup (\cup_{k=1}^b\Delta_k')) & \hbox{for $t = -1$,} \\
(\mathbb R^3, \phi) & \hbox{for $ t < -1$.} \\
\end{array}%
\right.
\end{equation*}

It is proved in \cite{KSS} that the isotopy type of $\mathcal S(D)$ does not 
depend on choices of $\Delta_j$'s and $\Delta_k'$'s. We choose an orientation 
for the surface-link $\mathcal{S}(D)$ so that the orientation of the 
cross-section $\mathcal S(D)_0=\mathcal{S}(D) \cap \mathbb{R}^3_0$ induced 
from the chosen orientation of $\mathcal{S}(D)$ is coherent to the orientation 
of $\mathcal{BL}(D)$.

We call the oriented surface-link $\mathcal{S}(D)$ the {\it oriented 
surface-link associated with $D$}. It is easily seen that $D$ is a marked 
graph diagram associated with the oriented surface-link $\mathcal{S}(D)$.
In particular, an admissible diagram represents a closed surface-link while 
a non-admissible diagram represents a \textit{cobordism} between the positive 
and negative resolutions. 

\begin{example}
The oriented marked graph diagram below represents the pictured cobordism 
between the unknot and unlink of two components depicted by the broken surface 
diagram below.
\[\raisebox{0.3in}{\includegraphics{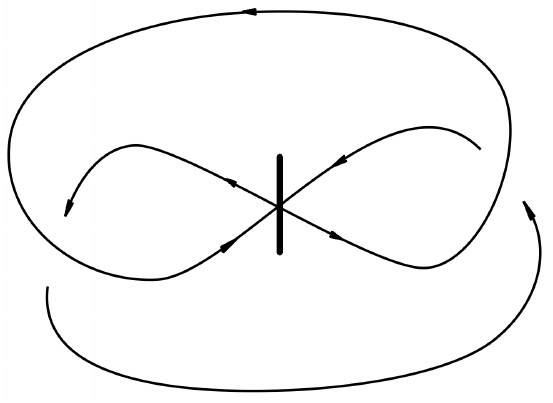}}\quad
\includegraphics{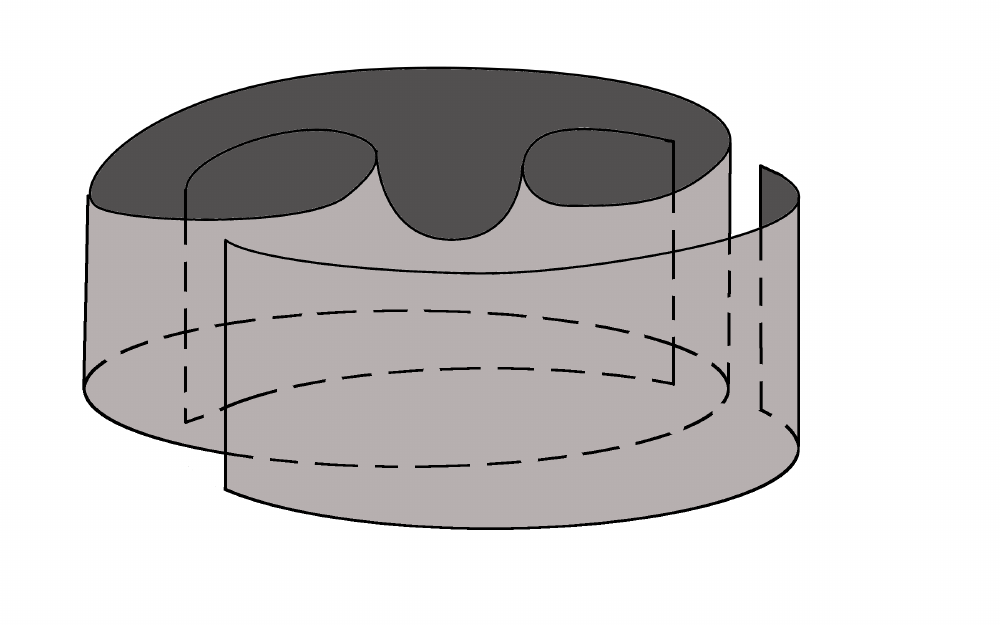}\]
Since the marked graph diagram is admissible, this cobordism can be capped 
off with disks to obtained a sphere in $\mathbb{R}^4$, in this case an 
unknotted sphere.
\end{example}

Two oriented marked graph diagrams represent ambient isotopic surface-links 
if and only if they are related by the following local moves, known as
\textit{Yoshikawa moves}, in addition to the usual oriented Reidemeister
moves; see \cite{KJL} for more.
\[\scalebox{1.4}{\includegraphics{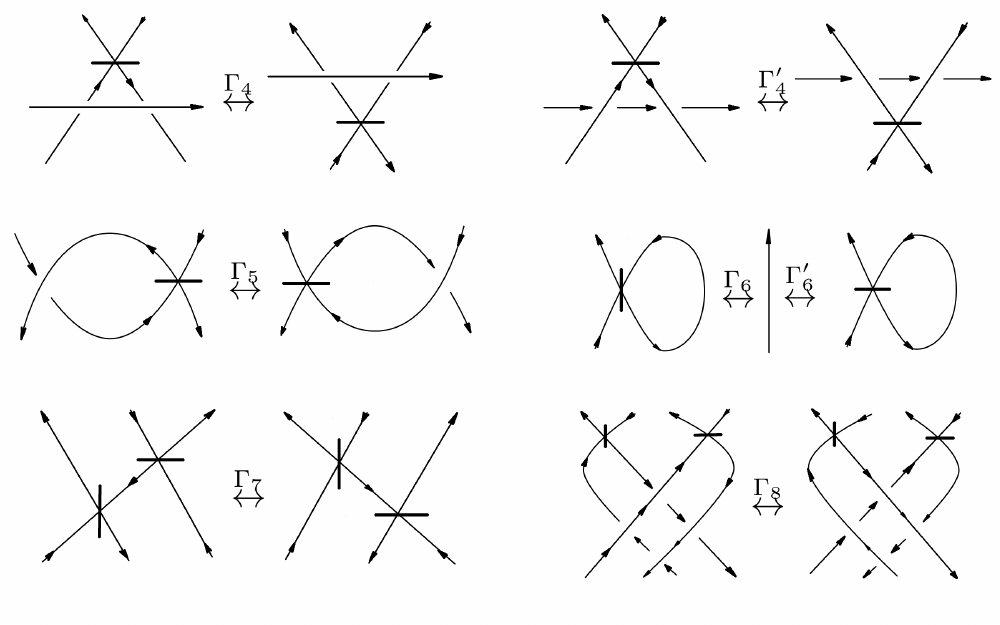}}\]

\section{\large\textbf{Biquasile Colorings of Surface-Links}}\label{SB}

We begin this section with a definition from \cite{dsn2}.

\begin{definition}
A \textit{biquasile} is a set $X$ with six binary operations 
$\ast,\backslash^{\ast},/^{\ast},\cdot, \backslash$and $/$ satisfying the axioms
\begin{itemize}
\item[(i)] For all $x,y\in X$ we have
\[
\begin{array}{rcccl}
y\backslash^{\ast}(y\ast x) & = & x & = & (x\ast y)/^{\ast} y\\
y\backslash(y\cdot x) & = & x & = & (x\cdot y)/ y\\
\end{array}
\]
and
\item[(ii)]
\[\begin{array}{rcl}
a\ast(x\cdot [y\ast(a\cdot b)]) 
& = & (a\ast[x\cdot y])\ast(x\cdot [y\ast([a\ast(x\cdot y)]\cdot b)]) \\
y\ast([a\ast (x\cdot y)]\cdot b) 
& = & (y\ast[a\cdot b])\ast([a\ast (x\cdot [y\ast(a\cdot b)])]\cdot b).
\end{array}\]
\end{itemize}
Axiom (i) says that $X$ forms a \textit{quasigroup} under both the $\ast$
and $\cdot$ operations (hence ``biquasile'') with left and right inverse 
operations $\backslash^{\ast},\backslash$ and $/^{\ast},/$ respectively (at
least in the finite case; see remark \ref{QG} below)
and axiom (ii) specifies the
relationship between the operations, like a complicated form of distributivity.
\end{definition}

\begin{remark}\label{QG}
The conditions in axiom (i) arising from the Riedemeister moves are really 
the requirements that for every $y\in X$, the maps 
$x\mapsto x\ast y, y\ast x, x\cdot y$ and $y\cdot x$ are invertible with
inverse maps given by $x\mapsto x/^{\ast}y$ etc. If $X$ is a finite set 
then the inverse maps commute with the original maps and we also have
\[
\begin{array}{rcccl}
y\ast (y\backslash^{\ast} x) & = & x & = & (x/^{\ast} y)\ast x\\
y\cdot(y\backslash x) & = & x & = & (x/y)\cdot y;
\end{array}
\]
however, for infinite $X$ these conditions are not imposed \textit{a priori}
in the biquasile definition. In practice, we have not considered biquasiles
which do not also satisfy these additional conditons; it may be of interest 
to do so in the future.
\end{remark}

\begin{example}
Let $X$ be a module over $\mathbb{Z}[d^{\pm 1}, s^{\pm 1}, n^{\pm 1}]$. Then
$X$ is a biquasile under the operations
\[x\ast y= -dsn^2x+ ny\quad \mathrm{and}\quad 
x\cdot y = dx+ sy
. \]
We have left inverse operations given by
\[x\backslash^{\ast} y = n^{-1} x+ dsn y 
\quad \mathrm{and}\quad 
x\backslash y = -ds^{-1}x+s^{-1}y
\]
and right inverse operations
\[x\backslash^{\ast} y = -d^{-1}s^{-1}n^{-2} x+ d^{-1}s^{-1}n^{-1} y
\quad \mathrm{and}\quad 
x/y = -ds^{-1}x+s^{-1}y.\]
Biquasiles of this sort are called \textit{Alexander biquasiles}.
\end{example}

\begin{example}
For finite biquasiles we can specify the biquasile structure with
a block matrix encoding the operation tables of $\cdot$ and $\ast$.
For example, the Alexander biquasile structure on $X=\mathbb{Z}_4$
with $d=s=1$ and $n=3$ has operations
\[x\ast y=-dsn^2x+ny=-(1)(1)(3^2)x+3y=3x+3y
\quad\mathrm{and}\quad 
x\cdot y=dx+sy=x+y
\]
so we have operation tables (using 4 for the class of 0 mod 4 since we number
our rows and columns starting with 1):
\[
\begin{array}{r|rrrr}
\ast & 1 & 2 & 3 & 4\\\hline
1 & 2 & 1 & 4 & 3 \\
2 & 1 & 4 & 3 & 2 \\ 
3 & 4 & 3 & 2 & 1 \\
4 & 3 & 2 & 1 & 4
\end{array}
\quad\mathrm{and}\quad
\begin{array}{r|rrrr}
\cdot & 1 & 2 & 3 & 4\\\hline
1 & 2 & 3 & 4 & 1 \\
2 & 3 & 4 & 1 & 2 \\ 
3 & 4 & 1 & 2 & 3 \\
4 & 1 & 2 & 3 & 4
\end{array}
\]
or in matrix form
\[\left[
\begin{array}{rrrr|rrrr}
2 & 1 & 4 & 3 & 2 & 3 & 4 & 1 \\ 
1 & 4 & 3 & 2 & 3 & 4 & 1 & 2 \\
4 & 3 & 2 & 1 & 4 & 1 & 2 & 3 \\  
3 & 2 & 1 & 4 & 1 & 2 & 3 & 4 
\end{array}
\right].
\]
\end{example}

\begin{definition}
Let $X$ be a biquasile. Then a \textit{biquasile coloring} of an oriented 
marked graph 
diagram $D$ is an assignment of elements of $X$ to the regions in the planar 
complement of $D$ such that at every classical crossing or marked vertex, we 
have the following:
\[\includegraphics{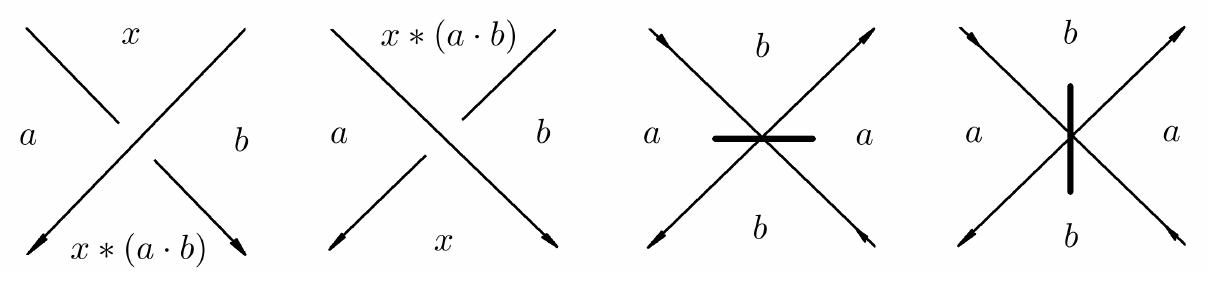}\]
\end{definition}

\begin{theorem}
If an oriented marked graph diagram $D$ has a biquasile coloring by a 
biquasile $X$ before
a Yoshikawa move taking $D$ to $D'$, there is a unique biquasile coloring of 
$D'$ which agrees with the given coloring of $D$ outside the neighborhood of 
the move.
\end{theorem}

\begin{proof}
We verify for each of the moves $\Gamma_4$, $\Gamma_4'$, $\Gamma_5$, $\Gamma_6$, 
$\Gamma_6'$, $\Gamma_7$ and $\Gamma_8$; see \cite{dsn2} for the case of the
classical Reidemeister moves.
\[\scalebox{0.9}{\includegraphics{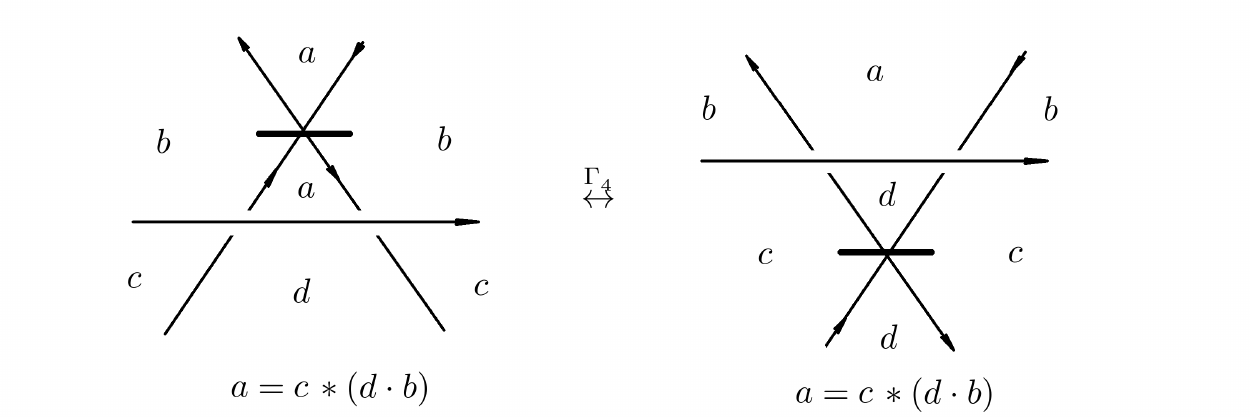}}\]
\[\scalebox{0.9}{\includegraphics{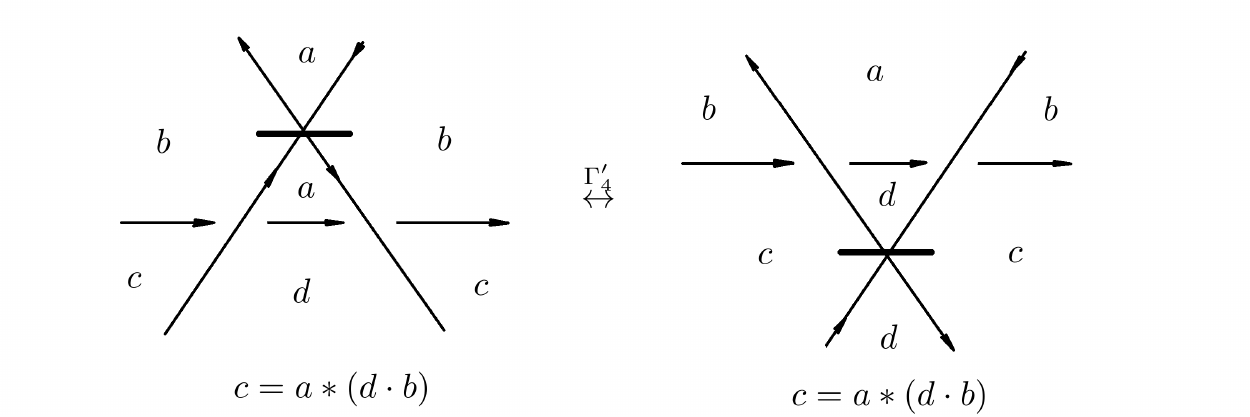}}\]
\[\scalebox{0.9}{\includegraphics{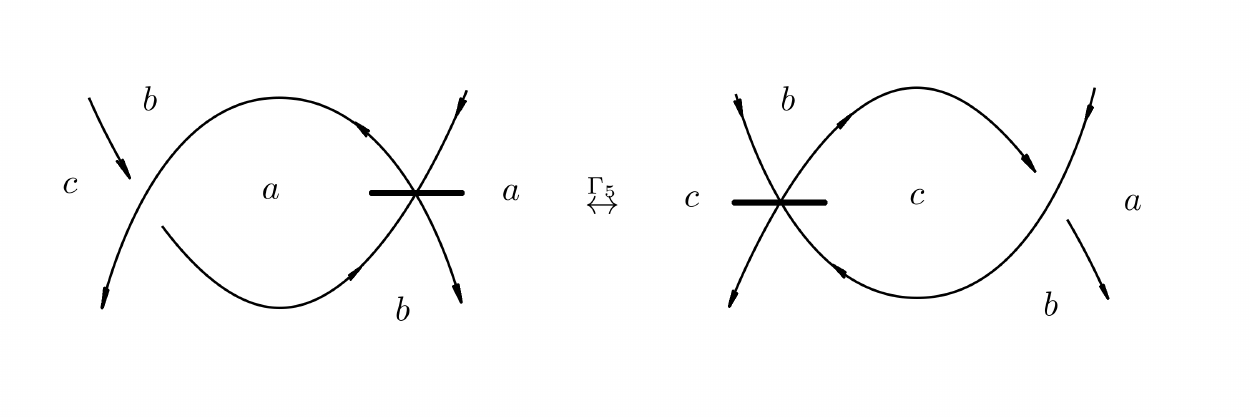}}\]
\[\scalebox{0.9}{\includegraphics{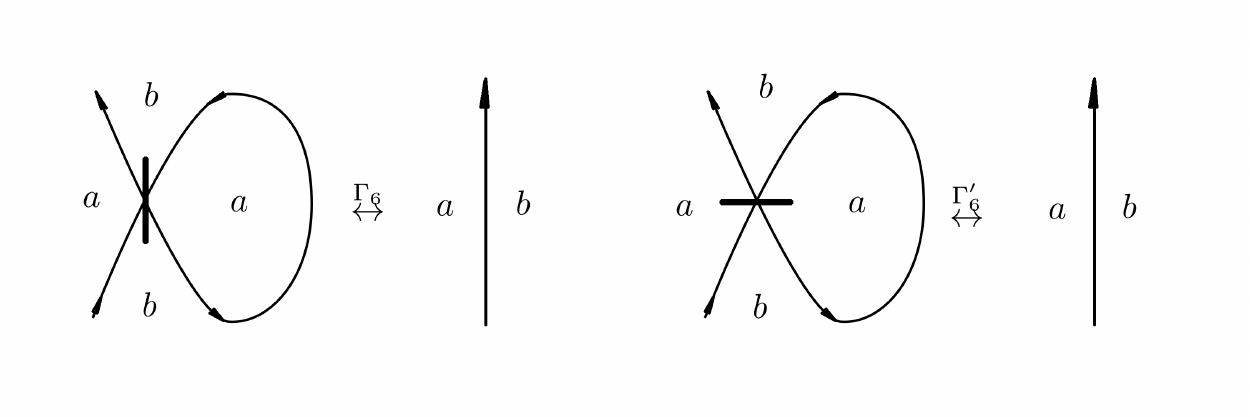}}\]
\[\scalebox{0.9}{\includegraphics{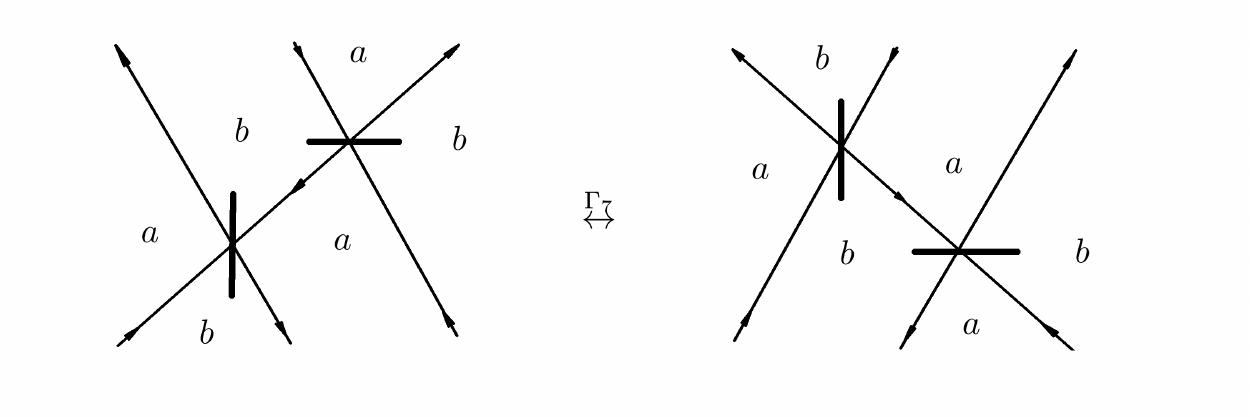}}\]
\[\scalebox{0.9}{\includegraphics{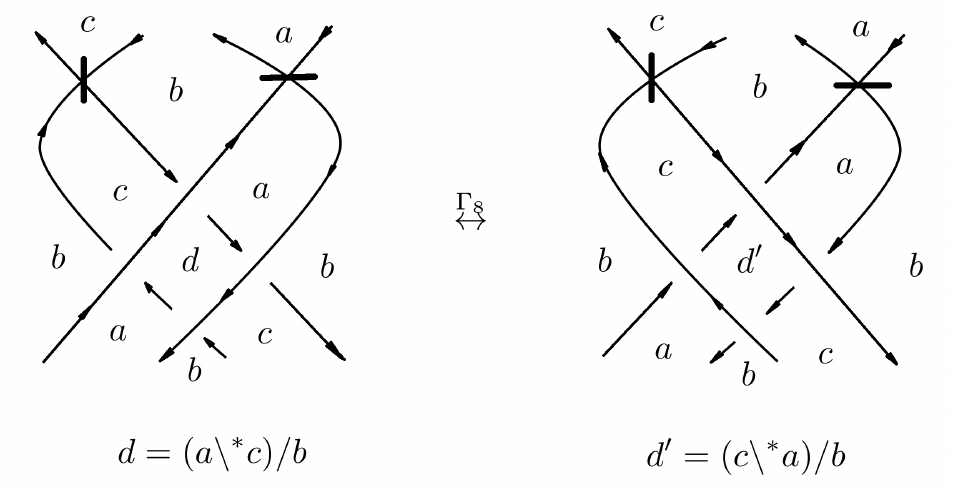}}\]
\end{proof}

\begin{corollary}
For any finite biquasile $X$, the number $\Phi_X^{\mathbb{Z}}(L)$ of biquasile 
colorings of a marked graph diagram $D$ representing an oriented surface-link 
$L$ is a surface-link invariant. We call this invariant the \textit{biquasile 
counting invariant}.
\end{corollary}

\begin{example}
Let $X$ be the Alexander biquasile structure on $\mathbb{Z}_7$ with $d=2$, 
$s=3$ and $n=4$, so we have
\[\begin{array}{rcl}
x \ast y & = & 5x+4y \\
x\cdot y & = & 2x+3y.
\end{array}\] 
Consider the marked graph diagram below which represents the surface-link
$6_1$. 
\[\includegraphics{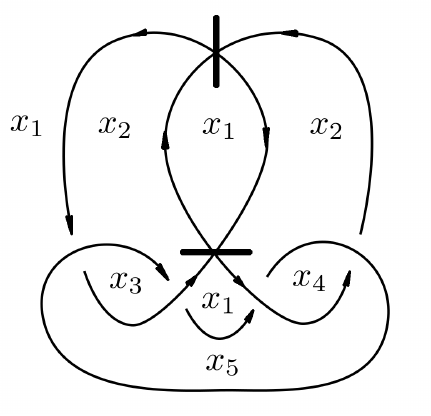}\]
Assigning variables $x_1\dots,x_6$ to the regions, we obtain system 
of coloring equations
\[\begin{array}{rcl}
x_1\ast(x_5\cdot x_2) & = & x_4 \\
x_4\ast(x_5\cdot x_2) & = & x_1 \\
x_1\ast(x_5\cdot x_2) & = & x_3 \\
x_3\ast(x_5\cdot x_2) & = & x_1 \\
\end{array}\leftrightarrow
\begin{array}{rcl}
5x_1+4(2x_5+3 x_2) & = & x_4 \\
5x_4+4(2x_5+3 x_2) & = & x_1 \\
5x_1+4(2x_5+3 x_2) & = & x_3 \\
5x_3+4(2x_5+3 x_2) & = & x_1 \\
\end{array}\leftrightarrow
\begin{array}{rcl}
5x_1+5x_2+6x_4+x_5 & = & 0 \\
6x_1+5x_2+5x_4+x_5 & = & 0 \\
5x_1+5x_2+6x_3+x_5 & = & 0 \\
6x_1+5x_2+5x_3+x_5 & = & 0 \\
\end{array}
\]
Then we have coefficient matrix
\[\left[\begin{array}{rrrrr}
5 & 5 & 0 & 6 & 1 \\
6 & 5 & 0 & 5 & 1\\
5 & 5 & 6 & 0 & 1\\
6 & 5 & 5 & 0 & 1
\end{array}\right]
\stackrel{\mathrm{row\ moves\ over\ }\mathbb{Z}_7}{\longrightarrow}
\left[\begin{array}{rrrrr}
1 & 2 & 0 & 2 & 6\\
0 & 1 & 3 & 2 & 3\\
0 & 0 & 1 & 6 & 0 \\
0 & 0 & 0 & 0 & 0 \\
\end{array}\right]
\]
so the kernel has dimension $2$ and the space of colorings has $7^2=49$ 
$X$-colorings, i.e. $\Phi_X^{\mathbb{Z}}(6_1)=49$.
Since the corresponding unlinked object (a standard unlinked torus and 
sphere) 
\[\includegraphics{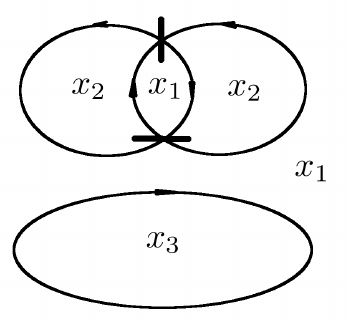}\]
has $7^3=343$ colorings by $X$, the biquasile counting invariant with respect
to this biquasile $X$ detects the nontrivality of the surface-link $6_1$.
\end{example}

\begin{example}\label{ex:3}
We selected three biquasiles of order 3 
\[
X_1=\left[\begin{array}{rrr|rrr}%X[0]
1 & 2 & 3 & 1 & 2 & 3 \\
2 & 3 & 1 & 3 & 1 & 2 \\
3 & 1 & 2 & 2 & 3 & 1
\end{array}\right],\quad
X_2=\left[\begin{array}{rrr|rrr}%X[1]
1 & 2 & 3 & 1 & 2 & 3 \\
3 & 1 & 2 & 2 & 3 & 1 \\
2 & 3 & 1 & 3 & 1 & 2 
\end{array}\right],\quad
X_3=\left[\begin{array}{rrr|rrr}%X[16]
2 & 1 & 3 & 1 & 3 & 2 \\
1 & 3 & 2 & 3 & 2 & 1 \\
3 & 2 & 1 & 2 & 1& 3
\end{array}\right].
\]
and computed the counting invariant
for all of the orientable surface-links with $ch$-index up to 10 with the 
orientations as shown 
\[\scalebox{0.85}{\includegraphics{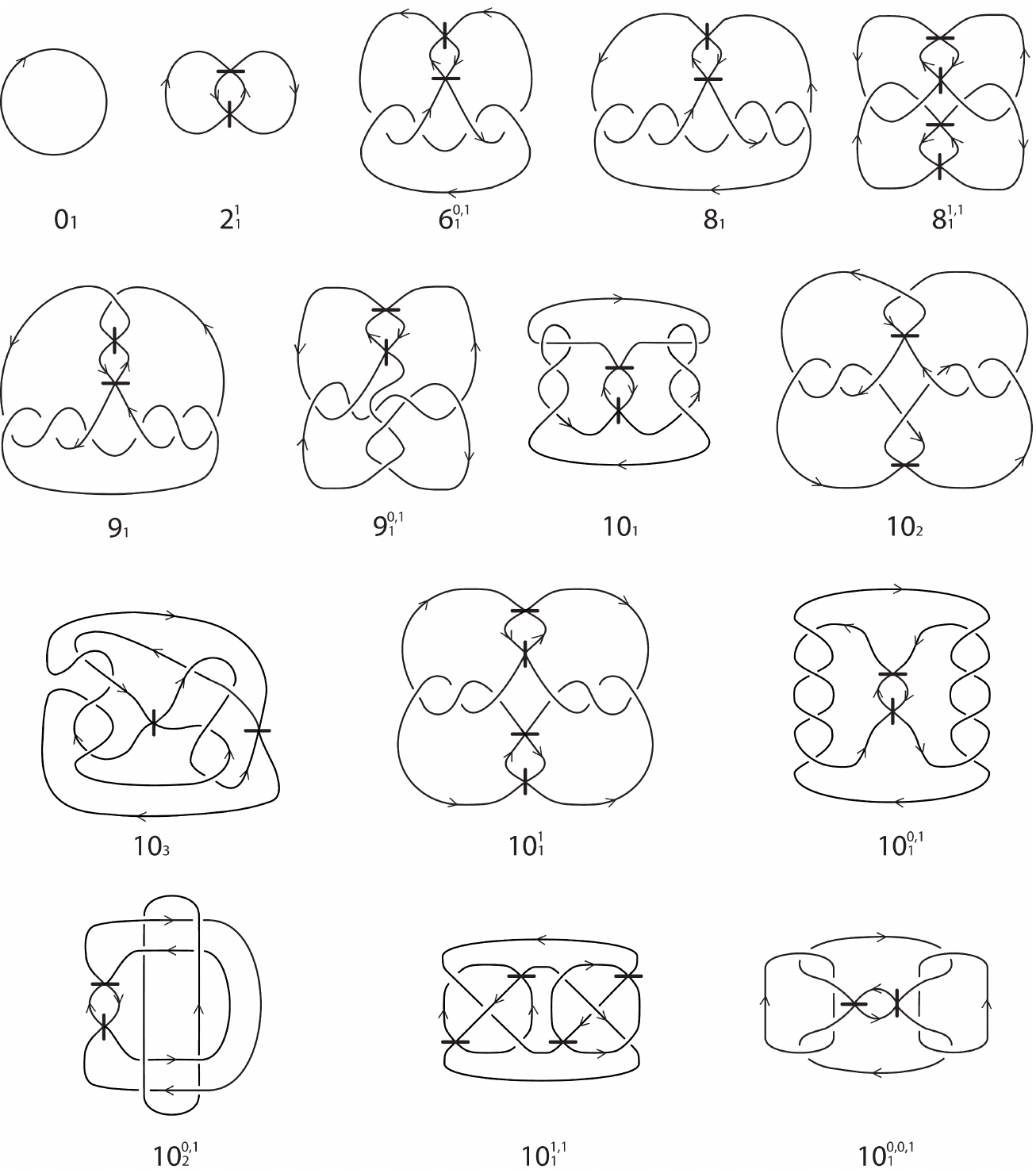}}\]
(see \cite{KJL1} for more) using our
\texttt{Python} code. The results are collected in the table.
%\[\begin{array}{|r|rrrrrrr|}\hline
%% & &  & &  &  &  &  \\ 
%L & 2_1 & 6_1^{0,1} & 8_1 & 8_1^{1,1} & 9_1  & 9_1^{0,1} & 10_1 \\ \hline
%\Phi_{X_1}^{\mathbb{Z}} & 9 & 9 & 27 & 9 & 27 & 3 & 9 \\
%\Phi_{X_2}^{\mathbb{Z}} & 9 & 27 & 9 & 27 & 3 & 9 & 9  \\
%\Phi_{X_3}^{\mathbb{Z}} & 9 & 27 & 9 & 27 & 9 & 0 & 9  \\\hline\hline
%% & &  & &  &  &  &  \\ 
%L & 10_2 & 10_3 & 10_1^1 & 10_1^{0,1} & 10_2^{0,1}  & 10_1^{1,1} & 10_1^{0,0,1} \\ \hline
%\Phi_{X_1}^{\mathbb{Z}} & 27 & 9 & 27 & 9 & 3 & 9 & 27  \\
%\Phi_{X_2}^{\mathbb{Z}} & 9 & 9 & 9 & 27 & 3 & 27 & 81  \\
%\Phi_{X_3}^{\mathbb{Z}} & 9 & 9 & 9 & 27 & 0 & 27 & 81 \\\hline
%\end{array}
%\]
\[\begin{array}{|r|rrrrrrrrrrrrrr|}\hline
% & &  & &  &  &  &  \\ 
L & 2_1 & 6_1^{0,1} & 8_1 & 8_1^{1,1} & 9_1  & 9_1^{0,1} & 10_1 %\\ \hline
L & 10_2 & 10_3 & 10_1^1 & 10_1^{0,1} & 10_2^{0,1}  & 10_1^{1,1} & 10_1^{0,0,1} \\ \hline 
\Phi_{X_1}^{\mathbb{Z}} & 9 & 9 & 27 & 9 & 27 & 9 & 9 & 27 & 9 & 27 & 9 & 3 & 9 & 27  \\%\\
\Phi_{X_2}^{\mathbb{Z}} & 9 & 27 & 9 & 27 & 3 & 27 & 9 & 9 & 9 & 9 & 27 & 3 & 27 & 81  \\ %\\
\Phi_{X_3}^{\mathbb{Z}} & 9 & 27 & 9 & 27 & 9 & 27 & 9 & 9 & 9 & 9 & 27 & 0 & 27 & 81 \\\hline  %\\\hline\hline
% & &  & &  &  &  &  \\ 
\end{array}
\]

\end{example}

\begin{example}
Biquasile counting invariants are sensitive to orientation-reversals, More 
precisely, reversing the orientation of the sphere component in the 
surface-link $L=9_1^{0,1}$ to obtain $L'$ 
as depicted results in different numbers of 
$X$-colorings for all three of the biquasiles in example \ref{ex:3}.
\[
\raisebox{-0.5in}{\includegraphics{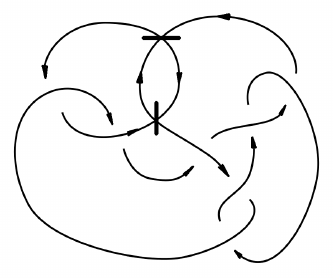}}
\quad 
\begin{array}{c}
\Phi_{X_1}^{\mathbb{Z}}(L)=9\\ 
\Phi_{X_2}^{\mathbb{Z}}(L)=27\\
\Phi_{X_3}^{\mathbb{Z}}(L)=27
\end{array}
\quad\quad
\raisebox{-0.5in}{\includegraphics{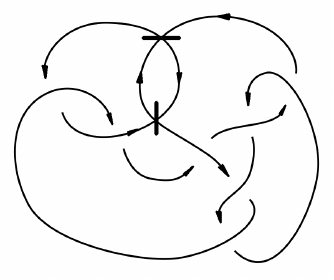}}
\quad 
\begin{array}{c}
\Phi_{X_1}^{\mathbb{Z}}(L')=3\\ 
\Phi_{X_2}^{\mathbb{Z}}(L')=9\\
\Phi_{X_3}^{\mathbb{Z}}(L')=0
\end{array}
 \]

\end{example}

\section{\large\textbf{Boltzmann Weight Enhancement}}\label{BW}

In \cite{cnn}, biquasile colorings of oriented link diagrams are enhanced with
\textit{Boltzmann weights}. In this section we extend this construction to
the case of oriented surface-links.

\begin{definition}
Let $X$ be a biquasile and $A$ an abelian group. Then an $A$-linear map 
$\phi:A[X^3]\to A$ from $A$-linear combinations of ordered triples of 
elements of $X$ to  $A$ is a \textit{Boltzmann weight} if for all
$x,y,a,b\in X$ we have
\begin{itemize}
\item[(i)]\[
\phi(x,a,a\backslash(x\backslash^{\ast}x))
=\phi(x,(x\backslash^{\ast} x)/b,b)
=0
\]
and
\item[(ii)]
\[\begin{array}{l}
\phi(x,a,b)+\phi(b,x\ast(a\cdot b),y)
+\phi(x\ast(a\cdot b),a,b\ast([x\ast(a\cdot b)]\cdot y)) \\
 = 
\phi(b,x,y) +
\phi(x,a,b\ast(x\cdot y))
+\phi(b\ast(x\cdot y), x\ast(a\cdot[b\ast(x\cdot y)]),y).
\end{array}.\]
\end{itemize}
\end{definition}

The Boltzmann weight definition collects the conditions required to make
the sum of $\phi(x,a,b)$ values at all crossings in a biquasile-labeled
oriented link diagram using the rule
\[\includegraphics{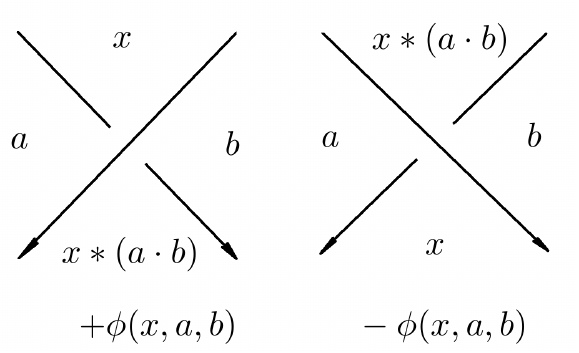}\]
unchanged by the oriented Reidemeister moves. Then for a classical oriented 
link $L$ the multiset
\[\Phi_X^{\phi,M}(L)=\left\{\sum_{\mathrm{crossing\ in \ } L_f}\pm\phi(x,a,b)\ |\ L_f\ X\ \mathrm{coloring\ of \ }L\right\}\]
is an invariant of links whose cardinality is the biquasile counting invariant.
The ``$M$'' here stands for ``multiset'' (as opposed to the polynomial
version of the invariant also defined \cite{cnn}) and $L_f$ is a biquasile 
coloring of the surface link diagram $L$ with coloring map $f$ assigning 
elements $f(r_j)\in X$ to each region $r_j$ in $L$. See \cite{cnn} for more.

\begin{lemma}\label{lem:bw}
Let $L$ be a marked graph diagram with a coloring by a biquasile $X$ and 
let $\phi:A[X^3]\to A$ be a Boltzmann weight for an abelian group
$A$. Then the sum 
\[\phi(L_f)=\sum_{\mathrm{crossing\ in \ } L_f}\pm\phi(x,a,b)\]
of contributions at the crossings and marked vertices
according to the rule
\[\includegraphics{jk-sn-11.pdf}\]
is unchanged by Yoshikawa moves.
\end{lemma}

\begin{proof}
The case of classical Reidemeister moves is covered in \cite{cnn}. 
We observe that for each of the additional moves
$\Gamma_4$, $\Gamma_4'$, $\Gamma_5$, $\Gamma_6$, 
$\Gamma_6'$, $\Gamma_7$ and $\Gamma_8$ the contributions on each side
of each move are equal.
\end{proof}

\begin{corollary}
Let $L$ be a marked graph diagram, $X$ a biquasile, $A$ an abelian group and
$\phi:A[X^3]\to A$ a Boltzmann weight. Then the multiset 
\[\Phi_X^{\phi,M}(L)=\{\phi(L_f)\ |\ L_f\ X-\mathrm{coloring\ of\ }L\}\]
where $\phi(L_f)$ is as defined in Lemma \ref{lem:bw} 
and the polynomial
\[\Phi_X^{\phi}(L)=\sum_{L_f\ X-\mathrm{coloring\ of\ }L} u^{\phi(L_f)}\]
are invariants of surface-links known as \textit{biquasile Boltzmann 
enhancements} of $\Phi_X^{\mathbb{Z}}$. 
\end{corollary}

\begin{proposition}
If $L$ is a marked graph diagram representing a cobordism between
oriented classical links $L_1$ and $L_2$ then we have inclusions
\[\Phi_X^{\phi,M}(L)\subset\Phi_X^{\phi,M}(L_1)
\quad\mathrm{and}\quad
\Phi_X^{\phi,M}(L)\subset\Phi_X^{\phi,M}(L_2)\]
for every Boltzmann weight $\phi$.
\end{proposition}

\begin{proof}
Every $X$-coloring of a marked graph diagram extends to an $X$-coloring of 
both of its resolutions.
Since the Boltzmann weight is determined by its values at classical crossings,
the Boltzmann weight at the level $t=0$ is the same as that for $t>0$ and
for $t<0$.
\end{proof}

\begin{corollary}
If $L$ is an marked graph diagram representing a closed surface-link, i.e.
a cobordism between unlinks, then $\Phi_X^{\phi}(L)=|\Phi_X^{\mathbb{Z}}(L)|$
for every Boltzmann weight $\phi$, i.e., the Boltzmann enhancement is trivial
for closed surface-links.
\end{corollary}

We also have the following easy observation:

\begin{proposition}
The biquasile counting invariant and its Boltzmann enhancement do not
detect the genus of the surface-link or cobordism.
\end{proposition}

\begin{proof}
We simply note that stabilization moves do not change either 
the number of colorings or the Boltzmann weight of a coloring:
\[\includegraphics{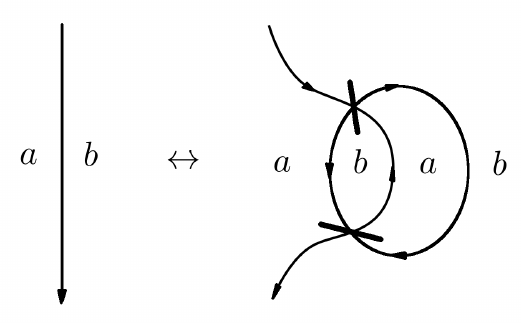}\]
\end{proof}

Our final example shows that biquasile Boltzmann enhancements \textit{can} 
detect knottedness of cobordisms. More precisely, a classical link diagram can 
be considered as a marked graph diagram without marked vertices, which 
corresponds to the trivial cobordism $L\times [0,1]$ between two copies of $L$. 

\begin{example}
Let $L2a1$ be the trivial cobordism between two Hopf links and let
$L$ be the cobordisms between two Hopf links determined 
by the marked graph diagram below: 
\[\includegraphics{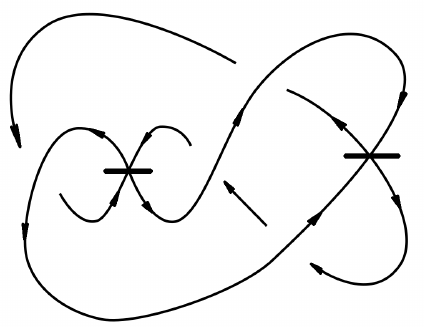}\]
Let $X$ be the biquasile with matrix
\[\left[\begin{array}{rr|rr}
1 & 2 & 2 & 1 \\
2 & 1 & 1 & 2
\end{array}\right],\]
let $R=\mathbb{Z}_5$ and let $\phi=\chi_{(2,1,2)}+\chi_{(2,2,1)}$. Then we have
\[\Phi_X^{\phi}(L)=4u+4\ne 4=\Phi_X^{\phi}(L2a1)\]
and the biquasile Boltzmann enhancement detects the difference between $L$ 
and $L2a1$.
\end{example}

\section{\large\textbf{Questions}}\label{Q}

We end with some questions for future research. 

A marked diagram which is not admissible, i.e., such that $\mathcal{L}_+$, 
$\mathcal{L}_-$ or both are not unlinks, defines a cobordism between 
$\mathcal{L}_+$ and $\mathcal{L}_-$. In particular, if a knot $K$ is slice then 
a cobordism exists between $K$ and an unknot. Can biquasile invariants be 
used to detect when a knot is not slice?

Currently biquasile invariants have only been defined for classical knots and 
links of dimensions 1 and 2. What about biquasile invariants for virtual knots
and links and virtual surface-links?

\bibliography{jk-sn-rev}{}
\bibliographystyle{abbrv}

\bigskip

\noindent
\textsc{Osaka City University Advanced Mathematical Institute \\ 
Osaka City University \\ 
Sugimoto, Sumiyoshi-ku, Osaka 558-8585, Japan
}

\medskip

\noindent
\textsc{Department of Mathematical Sciences \\
Claremont McKenna College \\
850 Columbia Ave. \\
Claremont, CA 91711}

\end{document}